\documentclass[12pt]{amsart}
\usepackage{amsmath,amssymb,amsbsy,amsfonts,amsthm,latexsym,amsopn,amstext,
                              amsxtra,euscript,amscd,url}

\begin{document}

%% \allowdisplaybreaks

%%%%%%%%%%%%%%%%%%%%%%%%%%%%%%%%%%%%%%%%%%%%%%%%%%%%%%%%%%%%%%%%%%%%%%
%% Use these to print supplementary material that is not to
%% be included in the final draft
\long\def\FOOTNOTE#1{ \ignorespaces}
\long\def\FOOTNOTE#1{\footnote{\emph{Added Note}: #1}}
%%%%%%%%%%%%%%%%%%%%%%%%%%%%%%%%%%%%%%%%%%%%%%%%%%%%%%%%%%%%%%%%%%%%%%

\newcommand{\TITLE}{Character Sums with Division Polynomials}
\newcommand{\TITLERUNNING}{Character sums with division polynomials}
\newcommand{\DATE}{\today}
\newcommand{\VERSION}{1}

%%%%%%%% Set Up Theorem-Style Formats %%%%%%%%%%%%%% 
\theoremstyle{plain} 
\newtheorem{theorem}{Theorem} 
\newtheorem{conjecture}[theorem]{Conjecture}
\newtheorem{proposition}[theorem]{Proposition}
\newtheorem{lemma}[theorem]{Lemma}
\newtheorem{corollary}[theorem]{Corollary}

\theoremstyle{definition}
% A * surpresses numbering, for example
% \newtheorem*{definition}{Definition}
\newtheorem*{definition}{Definition}

\theoremstyle{remark}
\newtheorem{remark}[theorem]{Remark}
\newtheorem{example}[theorem]{Example}
\newtheorem{question}[theorem]{Question}
\newtheorem*{acknowledgement}{Acknowledgements}

\def\BigStrut{\vphantom{$(^{(^(}_{(}$}} % Add space in tables

%%%%%%%% Set Up Environment for Notation %%%%%%%%%%%%%% 
% This is currently set to allow quite wide items to be defined 
\newenvironment{notation}[0]{%
  \begin{list}%
    {}%
    {\setlength{\itemindent}{0pt}
     \setlength{\labelwidth}{4\parindent}
     \setlength{\labelsep}{\parindent}
     \setlength{\leftmargin}{5\parindent}
     \setlength{\itemsep}{0pt}
     }%
   }%
  {\end{list}}

%%%%%%%% Set Up Environment for Parts in Theorems %%%%%%%%%%%%%% 
\newenvironment{parts}[0]{%
  \begin{list}{}%
    {\setlength{\itemindent}{0pt}
     \setlength{\labelwidth}{1.5\parindent}
     \setlength{\labelsep}{.5\parindent}
     \setlength{\leftmargin}{2\parindent}
     \setlength{\itemsep}{0pt}
     }%
   }%
  {\end{list}}
% Use \Part{(a)}, instead of \item[(a)], to ensure upright font 
\newcommand{\Part}[1]{\item[\upshape#1]}

%%%%%%% PROOF ENDING AT DISPLAYED EQUATION %%%%%%%% 
% If a proof ends with a displayed equation, put the command 
% \qedtag just before the final \] (or, e.g., the final \end{align}) 
% and put the command \EndProofAtDisplay just before the \end{proof}.
%
\newcommand{\EndProofAtDisplay}{\renewcommand{\qedsymbol}{}}
\newcommand{\qedtag}{\tag*{\qedsymbol}}
%%%%%%%%%%%%%%%%%%%%%%%%%%%%%%%%%%%%%%%%%%%%%%%%%%%

%%%%%%%%%%%%%%%%%%
% Greek Alphabet %
%%%%%%%%%%%%%%%%%%
\renewcommand{\a}{\alpha}
\renewcommand{\b}{\beta}
\newcommand{\g}{\gamma}
\renewcommand{\d}{\delta}
\newcommand{\e}{\epsilon}
\newcommand{\f}{\phi}
\newcommand{\fhat}{{\hat\phi}}
\renewcommand{\l}{\lambda}
\renewcommand{\k}{\kappa}
\newcommand{\lhat}{\hat\lambda}
\newcommand{\m}{\mu}
\renewcommand{\o}{\omega}
\renewcommand{\r}{\rho}
\newcommand{\rbar}{{\bar\rho}}
\newcommand{\s}{\sigma}
\newcommand{\sbar}{{\bar\sigma}}
\renewcommand{\t}{\tau}
\newcommand{\z}{\zeta}

\newcommand{\D}{\Delta}
\newcommand{\F}{\Phi}
\newcommand{\G}{\Gamma}

%%%%%%%%%%%%%%%%%%%%
% Fraktur Alphabet %
%%%%%%%%%%%%%%%%%%%%
\newcommand{\ga}{{\mathfrak{a}}}
\newcommand{\gb}{{\mathfrak{b}}}
\newcommand{\gc}{{\mathfrak{c}}}
\newcommand{\gd}{{\mathfrak{d}}}
\newcommand{\gk}{{\mathfrak{k}}}
\newcommand{\gm}{{\mathfrak{m}}}
\newcommand{\gn}{{\mathfrak{n}}}
\newcommand{\gp}{{\mathfrak{p}}}
\newcommand{\gq}{{\mathfrak{q}}}
\newcommand{\gI}{{\mathfrak{I}}}
\newcommand{\gJ}{{\mathfrak{J}}}
\newcommand{\gK}{{\mathfrak{K}}}
\newcommand{\gM}{{\mathfrak{M}}}
\newcommand{\gN}{{\mathfrak{N}}}
\newcommand{\gP}{{\mathfrak{P}}}
\newcommand{\gQ}{{\mathfrak{Q}}}

%%%%%%%%%%%%%%%%%%%%%%%%%
% Calligraphic Alphabet %
%%%%%%%%%%%%%%%%%%%%%%%%%
\def\Acal{{\mathcal A}}
\def\Bcal{{\mathcal B}}
\def\Ccal{{\mathcal C}}
\def\Dcal{{\mathcal D}}
\def\Ecal{{\mathcal E}}
\def\Fcal{{\mathcal F}}
\def\Gcal{{\mathcal G}}
\def\Hcal{{\mathcal H}}
\def\Ical{{\mathcal I}}
\def\Jcal{{\mathcal J}}
\def\Kcal{{\mathcal K}}
\def\Lcal{{\mathcal L}}
\def\Mcal{{\mathcal M}}
\def\Ncal{{\mathcal N}}
\def\Ocal{{\mathcal O}}
\def\Pcal{{\mathcal P}}
\def\Qcal{{\mathcal Q}}
\def\Rcal{{\mathcal R}}
\def\Scal{{\mathcal S}}
\def\Tcal{{\mathcal T}}
\def\Ucal{{\mathcal U}}
\def\Vcal{{\mathcal V}}
\def\Wcal{{\mathcal W}}
\def\Xcal{{\mathcal X}}
\def\Ycal{{\mathcal Y}}
\def\Zcal{{\mathcal Z}}

%%%%%%%%%%%%%%%%%%%%%%%%%%%%
% Blackboard Bold Alphabet %
%%%%%%%%%%%%%%%%%%%%%%%%%%%%
\renewcommand{\AA}{\mathbb{A}}
\newcommand{\BB}{\mathbb{B}}
\newcommand{\CC}{\mathbb{C}}
\newcommand{\FF}{\mathbb{F}}
\newcommand{\GG}{\mathbb{G}}
\newcommand{\KK}{\mathbb{K}}
\newcommand{\NN}{\mathbb{N}}
\newcommand{\PP}{\mathbb{P}}
\newcommand{\QQ}{\mathbb{Q}}
\newcommand{\RR}{\mathbb{R}}
\newcommand{\ZZ}{\mathbb{Z}}

%%%%%%%%%%%%%%%%%%%%%%%%%%
% Boldface Math Alphabet %
%%%%%%%%%%%%%%%%%%%%%%%%%%
\def \bfa{{\mathbf a}}
\def \bfb{{\mathbf b}}
\def \bfc{{\mathbf c}}
\def \bfe{{\mathbf e}}
\def \bff{{\mathbf f}}
\def \bfF{{\mathbf F}}
\def \bfg{{\mathbf g}}
\def \bfn{{\mathbf n}}
\def \bfp{{\mathbf p}}
\def \bfr{{\mathbf r}}
\def \bfs{{\mathbf s}}
\def \bft{{\mathbf t}}
\def \bfu{{\mathbf u}}
\def \bfv{{\mathbf v}}
\def \bfw{{\mathbf w}}
\def \bfx{{\mathbf x}}
\def \bfy{{\mathbf y}}
\def \bfz{{\mathbf z}}
\def \bfX{{\mathbf X}}
\def \bfU{{\mathbf U}}
\def \bfmu{{\boldsymbol\mu}}

%%%%%%%%%%%%%%%%%%%%%%%%
% Barred Math Alphabet %
%%%%%%%%%%%%%%%%%%%%%%%%
\newcommand{\Gbar}{{\bar G}}
\newcommand{\Kbar}{{\bar K}}
\newcommand{\kbar}{{\bar k}}
\newcommand{\Obar}{{\bar O}}
\newcommand{\Pbar}{{\bar P}}
\newcommand{\Rbar}{{\bar R}}
\newcommand{\Qbar}{{\bar{\QQ}}}

%%%%%%%%%%%%%%%%%%%%%%%%%%%%%%
% Miscellaneous New Commands %
%%%%%%%%%%%%%%%%%%%%%%%%%%%%%%

\newcommand{\Aut}{\operatorname{Aut}}
\newcommand{\Ctilde}{{\tilde C}}
\newcommand{\defeq}{\stackrel{\textup{def}}{=}}
\newcommand{\Disc}{\operatorname{Disc}}
\renewcommand{\div}{\operatorname{div}}
\newcommand{\Div}{\operatorname{Div}}
\newcommand{\et}{{\text{\'et}}}  %% subscript for etale
\newcommand{\Etilde}{{\tilde E}}
\newcommand{\End}{\operatorname{End}}
\newcommand{\EDS}{\mathsf{W}}
\newcommand{\Fix}{\operatorname{Fix}}
\newcommand{\Frob}{\operatorname{Frob}}
\newcommand{\Gal}{\operatorname{Gal}}
\newcommand{\GCD}{{\operatorname{GCD}}}
\renewcommand{\gcd}{{\operatorname{gcd}}}
\newcommand{\GL}{\operatorname{GL}}
\newcommand{\hhat}{{\hat h}}
\newcommand{\Hom}{\operatorname{Hom}}
\newcommand{\Ideal}{\operatorname{Ideal}}
\newcommand{\Image}{\operatorname{Image}}
\newcommand{\Jac}{\operatorname{Jac}}
\newcommand{\longhookrightarrow}{\lhook\joinrel\relbar\joinrel\rightarrow}
\newcommand{\MOD}[1]{~(\textup{mod}~#1)}
\renewcommand{\pmod}{\MOD}
\newcommand{\Norm}{\operatorname{N}}
\newcommand{\NS}{\operatorname{NS}}
\newcommand{\notdivide}{\nmid}
\newcommand{\OK}[1]{\Ocal_{K,#1}^{\,\sharp}}
\newcommand{\longonto}{\longrightarrow\hspace{-10pt}\rightarrow}
\newcommand{\onto}{\rightarrow\hspace{-10pt}\rightarrow}
\newcommand{\ord}{\operatorname{ord}}
\newcommand{\Parity}{\operatorname{Parity}}
\newcommand{\pequiv}{\stackrel{pr}{\equiv}}
\newcommand{\Pic}{\operatorname{Pic}}
\newcommand{\Prob}{\operatorname{Prob}}
\newcommand{\Proj}{\operatorname{Proj}}
\newcommand{\rank}{\operatorname{rank}}
\newcommand{\res}{\operatornamewithlimits{res}}
\newcommand{\Resultant}{\operatorname{Resultant}}
\renewcommand{\setminus}{\smallsetminus}
\newcommand{\stilde}{{\tilde\sigma}}
\newcommand{\sign}{\operatorname{Sign}}
\newcommand{\Spec}{\operatorname{Spec}}
\newcommand{\Support}{\operatorname{Support}}
\newcommand{\tors}{{\textup{tors}}}
\newcommand{\Tr}{\operatorname{Tr}}
\newcommand{\Trace}{\operatorname{Tr}}
\newcommand\W{W^{\vphantom{1}}}
\newcommand{\Wtilde}{{\widetilde W}}
\newcommand{\<}{\langle}
\renewcommand{\>}{\rangle}
\newcommand{\semiquad}{\hspace{.5em}}

\newcommand{\CS}[2]{\genfrac(){}{}{#1}{#2}_3}  % cubic residue symbol 
\newcommand{\FCS}[2]{\genfrac[]{}{}{#1}{#2}_3}  % cubic residue symbol in F_k
\newcommand{\LS}[2]{\genfrac(){}{}{#1}{#2}}  % Legendre symbol 
\newcommand{\FLS}[2]{\genfrac[]{}{}{#1}{#2}}  % Legendre symbol in F_k
\renewcommand{\SS}[2]{\genfrac(){}{}{#1}{#2}_6}  % sextic residue symbol 

\let\hw\hidewidth

\hyphenation{para-me-tri-za-tion}

\newcommand{\comm}[1]{\marginpar{%
\vskip-\baselineskip %raise the marginpar a bit
\raggedright\footnotesize
\itshape\hrule\smallskip#1\par\smallskip\hrule}}

\def\({\left(}
\def\){\right)}
\def\<{\langle}
\def\>{\rangle}
\def\fl#1{\left\lfloor#1\right\rfloor}
\def\rf#1{\left\lceil#1\right\rceil}
\def\e{\mathbf{e}}

%%%%%%%%%%%%%%%  Topmatter %%%%%%%%%%%%%%%%%%

\title[\TITLERUNNING]{\TITLE}
%% \date{\DATE, Draft \#\VERSION}
\date{\today}
\author{Igor E Shparlinski}
\address{Department of Computing, Macquarie University, 
North Ryde, Sydney, NSW 2109, Australia}
\email{igor@comp.mq.edu.au} 

\author{Katherine E. Stange} 
\address{Department of Mathematics, Stanford University, 450 Serra Mall, Building 380, Stanford, CA 94305, USA}
\email{stange@math.stanford.edu}



%%%%%%%%%%%%%%%%%%%%%%%%%%%%%%%%%%%%%%%%%%%%%%%%%%%%%%%%%%%%%%%%%%%%%%
%%% Text (non-TeX) Abstract
%%%
%%%%%%%%%%%%%%%%%%%%%%%%%%%%%%%%%%%%%%%%%%%%%%%%%%%%%%%%%%%%%%%%%%%%%%

\begin{abstract} 
We obtain nontrivial estimates of quadratic character sums of 
division polynomials $\Psi_n(P)$, $n=1,2, \ldots$, 
evaluated at a given point $P$ on an elliptic curve over a finite field
of $q$ elements.  
Our bounds are nontrivial if the order of $P$ 
is at least $q^{1/2 + \varepsilon}$ for some fixed $\varepsilon > 0$.
This work is motivated by an open 
question about statistical indistinguishability of some cryptographically 
relevant sequences which has recently been brought up by 
K.~Lauter and the second author. 
\end{abstract}

\maketitle

\section{Division Polynomials and Character Sums}

Let $E$ be an elliptic curve over a finite field $\FF_q$ of characteristic $p>3$.  Denote by $E(\FF_q)$ the group of
points of $E$ defined over $\FF_q$.  We refer to~\cite{Silv1} for
background on elliptic curves. 

Let $\Psi_n$ be the $n$-th division polynomial for 
positive integers $n$.  
For a given point $P \in E(\FF_q)$, the sequence $\Psi_n(P)$ is often called an \emph{elliptic divisibility sequence}.  It satisfies the following recurrence relation \cite[Exercise~3.34]{Silv1}
\begin{equation}
\begin{split}
\label{eqn:rec}
\Psi_{h+i}(P)\Psi_{h-i}(P)&\Psi_{j}(P)^2 \\
+~\Psi_{i+j}(P)&\Psi_{i-j}(P)\Psi_{h}(P)^2\\
+~&\Psi_{j+h}(P)\Psi_{j-h}(P)\Psi_{i}(P)^2
=0
\end{split}
\end{equation}

By definition, $\Psi_n(P) = 0$ if and only if $[n]P = 0$.  Further, the sequence $\Psi_n(P)$ is necessarily periodic with some period $T$ and $T$ is always a multiple  
of the order of $P$ (see Lemma~\ref{lem:parper} below).  For 
 background on elliptic divisibility sequences, 
see~\cite{EvdPSW,Swar,Ward}.

Note that  elliptic divisibility sequences can be viewed 
as a generalisation of Lucas sequences. We recall that 
Lucas sequences (of the first kind) are sequences satisfying a recurrence of the form
\[
L_n = a L_{n-1} + b L_{n-2}, \quad L_0=0, \quad L_1=1,
\]
in given coefficients $a$ and $b$.  Lucas sequences, including Fibonacci numbers, satisfy~\eqref{eqn:rec} after an appropriate scaling (multiplication of the $n$-th term by $\lambda^{n^2-1}$ for some 
$\lambda$); see~\cite[Exercise~3.34]{Silv1} and~\cite[Section~VI]{Ward}.   

In this paper, for a fixed point $P \in E(\FF_q)$, and 
an integer $N \le T$, 
we obtain nontrivial estimates of sums of the form
$$
S_P(N) = \sum_{n =1}^N \chi\(\Psi_n(P)\),
$$
where $\chi$ is the quadratic character of $\FF_q$
(as usual, we set  $\chi(0)=0$). 
Character sums with linear recurrence sequences have 
been studied in~\cite{Shp}.  See also~\cite[Chapter~5]{EvdPSW} for a survey of 
estimates of exponential and character sums with various recurrence
sequences.
However, to our knowledge, for  elliptic divisibility sequences no results have been obtained prior to this work.

\section{Motivation}

This question also has a cryptographic connection. 
In~\cite{LautStan} the following  \emph{elliptic divisibility sequence residue problem}  has been considered:  given two points $P, Q \in E(\FF_q)$ such that $Q \in \langle P \rangle$, $Q \neq \Ocal$ and $\ord(P) \ge 4$, calculate $\chi( \Psi_k(P) )$ for the smallest positive $k$ such that $Q = [k]P$.  To find $k$ given the points $P$ and $Q$ is the well-known \emph{elliptic curve discrete logarithm problem} and its assumed difficulty is the basis of elliptic curve cryptography.  To solve the residue problem it certainly suffices to solve the discrete logarithm problem.  However, it may be possible to solve the residue problem without first calculating $k$.  It is shown in~\cite[Theorem~1.1]{LautStan}
that solving either of these problems in subexponential time leads 
to a solution of the other in subexponential time. 
For perspective, the  calculation of $\chi( \Psi_{k+1}(P)/\Psi_k(P) )$ takes only polynomial time from $P$ and $Q$, and does not 
reveal $k$, see~\cite[Section~8]{LautStan}.  This has raised the general question of what can be said about the residuosity of $\Psi_n(P)$.
More specifically, it has been shown in~\cite{LautStan} that 
the difficulty of a certain distinguishability problem of 
cryptographic interest depends on the bias between the quadratic 
residues and nonresidues amongst consecutive 
terms of the sequence $\Psi_n(P)$, $n=1, \ldots, N$, 
which is in turn is equivalent to estimating the sums $S_P(N)$.

%

%Supposing that $T$ is the period of $\Psi_n(P)$, 
%then on account of equation \eqref{eqn:antisym}, 
%we have trivially that
%\[
%\sum_{n=1}^{T-1} \Psi_n(P) = 0.
%\]

%If $\chi$ is quadratic and $-1$ is a quadratic 
%non-residue, then $\chi(\Psi_n(P))$ is anti-symmetric.  
%So letting $S$ be the period of $\chi(\Psi_n(P))$, 
%in this case we obtain
%\[
%\sum_{n=1}^{S-1} \chi(\Psi_n(P)) = 0.
%\]
%In this case it is more natural to ask about the sum
%\[
%\sum_{n=1}^{\frac{S-1}{2}} \chi(\Psi_n(P)) = 0.
%\]

%%%%%%%%%%%%%%%%%%%%%%%%%%%%%%%%%%%%%%%%%%%%%%%%%%%%%%%%%%%%%%%%%%%%%%%%%
\section{Prerequisites concerning division polynomials}
%%%%%%%%%%%%%%%%%%%%%%%%%%%%%%%%%%%%%%%%%%%%%%%%%%%%%%%%%%%%%%%%%%%%%%%%%

We recall some classical results, the first of which describes the ratio 
$\Psi_{n+r}(P)/\Psi_n(P)$.

By~\cite[Theorem~8]{Silv2} (see also~\cite[Theorem~8.1]{Ward}), we have:

\begin{lemma}
\label{lem:parper}
Let $P \in E(\FF_q)$ be of order $r\ge 3$.  Then for all positive $s, k \in \ZZ$,
\[
\Psi_{sr+k}(P) = a^{ks}b^{s^2} \Psi_k(P),
\]
where $a$ and $b$ are given by
\[
a = \frac{\Psi_{r-2}(P)}{\Psi_{r-1}(P)\Psi_2(P)}, \qquad b = \frac{\Psi_{r-1}(P)^2\Psi_2(P)}{\Psi_{r-2}(P)}.
\]
\end{lemma}

%In particular, if $T$ is the period of $\Psi_n(P)$, 
%then we have that
%\begin{equation}
%\label{eqn:antisym}
%\Psi_{T-n}(P) = -\Psi_n(P)
%\end{equation}

Furthermore, by~\cite[Lemma~6]{Silv2}, we also have:

\begin{lemma}
\label{lem: nm}
Let $n$ and $m$ be positive integers.  Then
$$
\Psi_{nm}(P) = \Psi_{n}([m]P)\Psi_{m}(P)^{n^2}.
$$
\end{lemma}

We remark that in general, for $P \in E(\FF_q)$  of order $r\ge 3$, 
the period $T$ of the the sequence $\Psi_n(P)$ 
may be as large as $r(q-1)$, 
see~\cite[Corollary~9]{Silv2}. In turn, $r$ 
can be of order $q$ as well, for example, if $P$ is a generator 
of the cyclic group of points.

However, the following result,  that is immediate from Lemma~\ref{lem:parper}, 
shows that  the sequence $\chi\(\Psi_n(P)\)$ is 
of smaller period.

\begin{lemma}
\label{lem: per chi}
Let $P \in E(\FF_q)$ be of order $r\ge 3$. Then the sequence 
$\chi\(\Psi_n(P)\)$ 
is periodic with period which  is a divisor of  $R =2r$. 
\end{lemma}

Thus, we see from Lemma~\ref{lem: per chi} that 
bounds of character sums $S_P(N)$  are of interest only for
the values of $N \le R= 2r$.

%%%%%%%%%%%%%%%%%%%%%%%%%%%%%%%%%%%%%%%%%%%%%%%%%%%%%%%%%%%%%%%%%
\section{Prerequisites concerning character sums}
%%%%%%%%%%%%%%%%%%%%%%%%%%%%%%%%%%%%%%%%%%%%%%%%%%%%%%%%%%%%%%%%%

It is well-known  that for an elliptic curve $E$ over $\FF_q$ we have
\[
E(\FF_q) \sim \ZZ / M \ZZ \times \ZZ / L \ZZ
\]
for unique integers $M$ and $L$ satisfying $L \mid M$.  The point $P$ and $Q$ are called \emph{echelonized generators} if $P$ has order $M$, $Q$ has order $L$ and any point of $E(\FF_q)$ can be written in the form $mP + \ell Q$ with $1 \leq m \leq M$ and $1 \leq \ell \le L$.

Let $\Omega = \Hom( E(k),\CC^* )$ be the group of characters on $E(k)$; this is given explicitly by
\[
\Omega = \{ \e_{M}(am) \e_{L}(b\ell)~:~0 \leq a < M, 0 \leq b < L \},
\]
where for a positive integer $K$, we define
\[
\e_K(z) = \exp\(2\pi i z / K\).
\]

The following multiplicative analogue of a result of~\cite{KohShp} 
is essentially~\cite[Proposition~1]{Chen}, 
which in turns comes from~\cite{Perret} (note that in~\cite{Chen}
it is formulated only for prime fields but the proof extends to 
arbitrary fields without any difficulties).

\begin{lemma}
\label{lem:charsum}
Let $\eta$ be a non-principal multiplicative character on $\FF_q^*$.  Let $\KK = \FF_q(E)$ be the function field of an elliptic curve $E$ over $\FF_q$, and $f \in \KK$ be of degree $d$ and such that $f \neq g^m$ for any function $g$ in the algebraic closure $\overline{\KK}$ of $\KK$
 and $m \mid q-1$.  Let $\omega \in \Omega$.
Then
\[
\left| \sum_{P \in E(\FF_q)}\hskip-25pt{\phantom{{M^\ell}}}^*\, \omega(P)\eta(f(P)) \right| \leq 2 d\sqrt{q}
\]
where $\sum^*$ indicates that the sum is over $P \in E(\FF_q)$ such that $f(P) \neq \infty$.
\end{lemma}

\begin{lemma}
\label{lem:subgroupsum}
Under the assumptions of Lemma~\ref{lem:charsum},  
let $H \subseteq E(\FF_q)$ be a subgroup.  Then
\[
\left| \sum_{P \in H}\hskip-18pt{\phantom{{M^\ell}}}^*\,
 \omega(P)\eta(f(P)) \right| \leq 2 d \sqrt{q}
\]
where $\sum^*$ indicates that the sum is over $P \in H$ such that $f(P) \neq \infty$.
\end{lemma}

\begin{proof}  Let $\Omega_H \subseteq \Omega$ be the subset of characters $\vartheta$ such that $H \subseteq \ker(\vartheta)$.  Then, $\Omega_H$ is dual to $E(\FF_q)/H$, so by the orthogonality property 
of characters of abelian groups, we have
\[
\frac{1}{\left| \Omega_H \right|} \sum_{\vartheta \in \Omega_H} \vartheta(P) = \left\{
\begin{array}{ll}
1 & P \in H, \\
0 & P \notin H.
\end{array} \right.  
\]
Therefore
\begin{align*}
 \sum_{P \in H}\hskip-18pt{\phantom{{M^\ell}}}^*\, \omega(P)\eta(f(P))
&=  \frac{1}{\left| \Omega_H \right|} \sum_{P \in E(\FF_q)}
\hskip-22pt{\phantom{{M^\ell}}}^*\,
 \sum_{\vartheta \in \Omega_H} \vartheta(P) \omega(P)\eta(f(P)) \\
&=  \frac{1}{\left| \Omega_H \right|} \sum_{\vartheta \in \Omega_H} \left( \sum_{P \in E(\FF_q)}\hskip-22pt{\phantom{{M^\ell}}}^*\, \vartheta\cdot\omega(P)\eta(f(P)) \right). \\
\end{align*}
Applying  Lemma~\ref{lem:charsum}, we obtain the desired result.
\end{proof}

%%%%%%%%%%%%%%%%%%%%%%%%%%%%%%%%%%%%%%%%%%%%%%%%%%%%%%%%%%
\section{Main results}
%%%%%%%%%%%%%%%%%%%%%%%%%%%%%%%%%%%%%%%%%%%%%%%%%%%%%%%%%%

Here we estimate the incomplete sum $S_P(N)$. Following 
the standard approach we start with estimates of complete 
sums twisted with an additive character.

As before, let  $R = 2r$  where $r$ is the order of $P$.
Then for an integer $a$ we define the sums
$$
T_P(a) = \sum_{n=1}^R \chi(\Psi_n(P))\e_R(an).
$$
which can be of independent interest. 

\begin{theorem}
\label{thm:compl sums} For any integer $a$, we have
\[
T_P(a)=O\(R^{5/6}q^{1/12} (\log q)^{1/3}\).
\]
\end{theorem}

\begin{proof}
Let $a \in \ZZ$.   Fix an integer $L\ge 3$ and 
let $\Lcal$ denote the set of odd primes $\ell$ such that $\ell<L$ and 
$\ell \nmid R$. Since $R$ has at most $O(\log R) = O(\log q)$ 
prime divisors we see, say, for  
\begin{equation}
\label{eqn: L large}
L \ge (\log q)^2
\end{equation}
and  sufficiently large $q$ we have
\begin{equation}
\label{eqn: l}
\#\Lcal   \geq \frac{L}{2\log L}.
\end{equation}

Let $\ell \in \Lcal$.  As $n$ runs through the all
residue  classes   modulo $R$, so does $\ell n$. 
Since both  sequences $\chi(\Psi_{n}(P))$ 
and $\e_R(a n)$, $n =1, 2, \ldots$,  
are periodic with period $R$,   we have
\[
T_P(a)
= 
\sum_{n=1}^R \chi(\Psi_{\ell n}(P))\e_R(a\ell n)
.
\]
We average over all choices of $\ell \in \Lcal$ and set
\[
W = \sum_{\ell \in \Lcal}  \sum_{n=1}^R  \chi(\Psi_{\ell n}(P))\e_R(a\ell n)  .
\]
Then we have
\begin{equation}
\label{eq:T and W}
T_P(a) = \frac{1}{\#\Lcal} W.
\end{equation}

To estimate $W$, we change the order of summation, and then apply the Cauchy inequality:
\[
\left| W \right|^2 \leq R \sum_{n=1}^R \left| \sum_{\ell \in \Lcal} \chi(\Psi_{\ell n}(P))\e_R(a\ell n) \right|^2.
\]
Now we apply Lemma~\ref{lem: nm}:
\begin{align*}
\left| W \right|^2 &\leq R \sum_{n=1}^R \left| \sum_{\ell \in \Lcal} \chi(\Psi_{\ell n}(P))\e_R(a\ell n) \right|^2 \\
 &=   R \sum_{n=1}^R  \left| \sum_{\ell \in \Lcal} \chi(\Psi_{\ell}(nP))\chi(\Psi_{n}(P)^{\ell^2})\e_R(a\ell n) \right|^2. 
  \end{align*}
Since $\chi$ is the quadratic character and $\ell$ is odd, we have 
\begin{equation}
\label{eq: no l}
\chi(\Psi_{n}(P)^{\ell^2}) = \chi(\Psi_{n}(P)).
\end{equation}
 Therefore, 
\begin{align*}
\left| W \right|^2   
 &\le R \sum_{n=1}^R \left| \chi(\Psi_n(P)) \right|^2 \left| \sum_{\ell \in \Lcal} \chi(\Psi_{\ell}(nP))\e_R(a\ell n) \right|^2\\
 &\le  R \sum_{n=1}^R \left| \sum_{\ell \in \Lcal} \chi(\Psi_{\ell}(nP))\e_R(a\ell n) \right|^2.
\end{align*}

Expanding the square and switching the order of summation again, 
we obtain
\begin{align*}
\left| W\right |^2 
&\leq  R \sum_{n=1}^R \sum_{\ell_1,\ell_2 \in \Lcal} \chi(\Psi_{\ell_1}(nP))\e_R(a\ell_1n)\chi(\Psi_{\ell_2}(nP))\e_R(-a\ell_2n)\\
&=  R\sum_{\ell_1,\ell_2 \in \Lcal} \sum_{n=1}^R \chi\(\Psi_{\ell_1}(nP)\Psi_{\ell_2}(nP)\)\e_R(a(\ell_1-\ell_2)n).
\end{align*}

We now turn to bounding the inner sum.  

For $\ell_1=\ell_2=\ell$, we have the trivial estimate
\[
\sum_{n=1}^R \chi(\Psi_\ell(nP)^2) < R .
\]

For $\ell_1\neq \ell_2$ we use Lemma~\ref{lem:subgroupsum}.  The degree of $\Psi_\ell(P)$ (considered as a function in the function field of  $E$)  is $(\ell^2-1)/2$, so the degree of 
$\Psi_{\ell_1}(P)\Psi_{\ell_2}(P)$ is 
$$
\frac{(\ell_1^2+\ell_2^2-2)}{2} < L^2-1.
$$  
It is also easy to see (by examining its zeros) 
that $\Psi_{\ell_1}(P)\Psi_{\ell_2}(P)$
is not a square of another function from the same function 
field. Since by Lemma~\ref{lem: per chi} we have $R\mid 2r$, 
we see that
$$
\sum_{n=1}^R \chi\(\Psi_{\ell_1}(nP)\Psi_{\ell_2}(nP)\)\e_R(a(\ell_1-\ell_2)n) 
= O(L^2q^{1/2}). 
$$
Thus, we obtain
\[
\left|W\right|^2 = O\(R^2  \# \Lcal  +   R  L^2 \sqrt{q}  (\# \Lcal)^2\). 
\]
Substituting this bound in~\eqref{eq:T and W} and using~\eqref{eqn: l}, 
we derive
\begin{align*}
T_P(a) & =   O\(R  (\# \Lcal)^{-1/2}  +  q^{1/4} R^{1/2} L \) \\
& =    O\(R L^{-1/2}(\log L)^{1/2}  +   q^{1/4} R^{1/2} L \).
\end{align*}

We no choose  $L = \fl{R^{1/3}q^{-1/6} (\log q)^{1/3}}$, 
thus~\eqref{eqn: L large} is satisfied, provided that $q$ is 
large enough which implies the desired estimate. 
\end{proof}

We remark that Theorem~\ref{thm:compl sums} is
nontrivial if $R \ge q^{1/2 + \varepsilon}$ for 
a fixed $\varepsilon > 0$ (we recall that the largest 
possible value of $R$ is of order $q$).

Now using the standard reduction between complete and 
incomplete sums, see~\cite[Section~12.2]{IwKow},  we obtain

\begin{corollary}
\label{cor: imcompl sums}
For any $N \le R$, we have, 
\[
S_P(N) =
O( R^{5/6}q^{1/12} (\log q)^{4/3}).
\]
\end{corollary}

%%%%%%%%%%%%%%%%%%%%%%%%%%%%%%%%%%%%%%%%%%%%%%%%%
\section{Comments}
%%%%%%%%%%%%%%%%%%%%%%%%%%%%%%%%%%%%%%%%%%%%%%%%%

In principle, our approach works for sums of multiplicative characters of 
arbitrary order  $d \mid q-1$. In this case, Lemma~\ref{lem: per chi} 
needs some obvious adjustments. Furthermore, the set $\Lcal$ in the proof 
of  Theorem~\ref{thm:compl sums} has to be chosen to consist
of primes $\ell \equiv \pm 1 \pmod d$, so~\eqref{eq: no l} still holds. 
For any fixed $d$ the final
result is the same, however its strength diminishes as $d$ grows, and
for example, for  characters of order $q-1$ leads only 
to a trivial estimate. Although we do not see any immediate 
cryptographic significance of such a result, obtaining nontrivial estimates
of character sums with arbitrary multiplicative characters is a
natural and interesting question. A related open question  is obtaining 
nontrivial estimates on similar sums of additive characters
of $\FF_q$. In this case, there is no natural analogue 
of~\eqref{eq: no l} and thus our approach does not apply at all.

Finally, we mention an algorithmic question which can be of 
cryptographic relevance. Given a black-box which for every 
integer $n$ outputs $\chi(\Psi_{n}(P))$, the question is to recover the 
``hidden'' point $P$. This admits several modifications 
depending whether the curve $E$ and the field $\FF_q$ are 
known or not. This question is analoguous to the more studied
cryptographic problem of recovering a hidden polynomial 
$f(X) \in \FF_q[X]$
given a black-box which outputs $\chi(f(n))$; see~\cite{RusShp}
and references therein.

\section*{Acknowledgement}

The authors would like to thank  Kristin Lauter who 
connected them together and stimulated their joint 
work on this paper. 
The authors are also grateful
to the Fields Institute 
for its support and stimulating atmosphere which 
led to the initiation of this work at 
the ``Fields Cryptography Retrospective Meeting'', 
Toronto, May 2009. 

During the preparation of this paper,  
I.~S. was supported in part by ARC
Grant~DP0881473 and K.~S. was supported  in part by 
NSF Fellowship 0802915 and NSERC PDF-373333.

\end{document}